\documentclass[11pt,reqno]{amsart}
\usepackage[all]{xy}
\usepackage{amssymb}
\usepackage{amsthm}
\usepackage{hyperref}
\usepackage[normalem]{ulem}
\usepackage{amsmath,mathtools}
\usepackage{amscd,enumitem}
\usepackage{caption}
\usepackage[justification=centering]{caption}
\usepackage{verbatim}
\usepackage{eurosym}
\usepackage{float}
\usepackage{color}
\usepackage{dcolumn}
\usepackage[mathscr]{eucal}
\usepackage[all]{xy}
\usepackage{bbm}
\usepackage{ytableau}
\usepackage[textheight=8.7in, textwidth=6.6in]{geometry}
\newtheorem*{conj*}{Conjecture}
\newtheorem{theorem}{Theorem}[section]

\theoremstyle{definition}
\newtheorem*{remark}{Remark}

\theoremstyle{plain}

\newtheorem{lemma}[theorem]{Lemma}

\newtheorem{corollary}[theorem]{Corollary}

\newcommand{\CC}{\mathcal{C}}

\newcommand{\leg}[2]{\left( \frac{#1}{#2} \right)}

\numberwithin{equation}{section}

\newtheoremstyle{example}
  {\topsep}   
  {\topsep}   
  {\normalfont}  
  {0pt}       
  {\bfseries} 
  {.}         
  {5pt plus 1pt minus 1pt} 
  {}          
\theoremstyle{example}
\newtheorem*{example}{Example}

\def\({\left(}
\def\){\right)}

\usepackage{centernot}

\begin{document}
\title{Zeros in the character tables of symmetric groups with an $\ell$-core index}
\author{Eleanor McSpirit and  Ken Ono}

\dedicatory{In memory of master representation theorist John McKay}

\address{Department of Mathematics, University of Virginia, Charlottesville, VA 22904}
\email{egm3zq@virginia.edu}
\email{ken.ono691@virginia.edu}

\keywords{Primary: character tables, hook lengths, partitions, symmetric groups}
\subjclass[2020]{20C30, 11P82, 05A17}

\thanks{E.M. acknowledges the support of a UVA Dean's Doctoral Fellowship. K.O. thanks  the Thomas Jefferson Fund and the NSF
(DMS-2002265 and DMS-2055118) for their support.}

\begin{abstract} Let $\CC_n =\left [\chi_{\lambda}(\mu)\right]_{\lambda, \mu}$ be the character table  for  $S_n,$ where the indices $\lambda$ and $\mu$ run over the $p(n)$ many integer partitions of $n.$ In this note we study $Z_{\ell}(n),$ the number of zero entries $\chi_{\lambda}(\mu)$ in $\CC_n,$ where $\lambda$ is an $\ell$-core partition of $n.$
For every prime $\ell\geq 5,$ we prove an asymptotic formula of the form
 $$
Z_{\ell}(n)\sim \alpha_{\ell}\cdot \sigma_{\ell}(n+\delta_{\ell})p(n)\gg_{\ell} n^{\frac{\ell-5}{2}}e^{\pi\sqrt{2n/3}},
 $$
 where $\sigma_{\ell}(n)$ is a twisted Legendre symbol divisor function, $\delta_{\ell}:=(\ell^2-1)/24,$ and
$1/\alpha_{\ell}>0$ is a normalization of the Dirichlet $L$-value $L\left (\leg{\cdot}{\ell},\frac{\ell-1}{2}\right).$
For primes $\ell$ and $n>\ell^6/24,$ we show 
 that $\chi_{\lambda}(\mu)=0$ whenever $\lambda$ and $\mu$ are both $\ell$-cores.
 Furthermore, if $Z^*_{\ell}(n)$ is the number of zero entries indexed by two $\ell$-cores, then for $\ell\geq 5$ we obtain the asymptotic
$$
Z^*_{\ell}(n)\sim \alpha_{\ell}^2 \cdot \sigma_{\ell}( n+\delta_{\ell})^2 \gg_{\ell} n^{\ell-3}.
$$
\end{abstract}

\maketitle

\section{Introduction and statement of results}

Let $\CC_n =\left [\chi_{\lambda}(\mu)\right]_{\lambda,\mu}$ be the usual character table
(for example, see \cite{J-K, Sagan, Stanley})
 for the symmetric group $S_n,$ where the indices $\lambda$ and $\mu$ both vary over the $p(n)$ many integer partitions of $n$. Confirming conjectures of Miller \cite{Mp}, Peluse and Soundararajan 
 \cite{P, P-S} recently proved that if $\ell$ is prime,
then almost all of the $p(n)^2$ entries in $\CC_n,$ as $n\rightarrow +\infty,$ are multiples of  $\ell.$ We note that Miller conjectured that the same conclusion holds for arbitrary prime powers, a claim which remains open.

In recent papers \cite{Mp, Mz}, Miller raised the problem of determining the limiting behavior of $Z(n),$  
 the number of zero entries in $\CC_n.$
Despite the remarkable theorem of Peluse and Soundararajan, little is known. 
Moreover, due to the rapid growth of $p(n),$ it is computationally infeasible to compute many  values of $Z(n).$  Consequently, there are no conjectures that are supported with substantial numerics. For example, is there  a limiting proportion for the zeros in $\CC_n$? Such a proportion would be given by the limit
$$
\lim_{n\rightarrow +\infty} \frac{Z(n)}{p(n)^2}.
$$
Limited numerics suggest that such a limit might exist, and might be  $\approx 0.36$ (see Table~3 of \cite{Mp}). However, this is a dubious guess at best. What's more, the simpler problem of determining whether
$\liminf_{n\rightarrow +\infty}Z(n)/p(n)^2>0$
also seems  to be out of reach.  In view of these difficulties, McKay \cite{He} posed a less ambitious problem; he asked for lower bounds arising from $\ell$-cores that illustrate the rapid growth of $Z(n).$
Here we answer this question, and
for primes $\ell \geq 5,$ we obtain asymptotic formulas for
\begin{equation}
Z_{\ell}(n):=\#\left \{ (\lambda, \mu) \ : \ \chi_{\lambda}(\mu)=0 \ \ {\text {\rm with  $\lambda$ an $\ell$-core}}\right\}.
\end{equation}

To this end, suppose that
 $\lambda=(\lambda_1, \lambda_2,\dots, \lambda_s)$ is a partition of $n.$ As is typical in the representation theory of symmetric groups, we
make use of $\ell$-core partitions, which are 
defined using {\it Young diagrams} of partitions,  the left-justified arrays of cells where the row lengths are the parts.  
 The {\it hook} for the
cell in position $(k,j)$ is the set of cells below or to the right of that cell, including the cell itself, and so its {\it hook length} $h_{\lambda}(k,j):=(\lambda_k-k)+(\lambda'_j-j)+1.$ Here $\lambda'_j$ is the number of boxes in the $j$th column of the diagram. We say that $\lambda$ is an {\it $\ell$-core partition}
 if none of its hook lengths are multiples of $\ell.$ 
If $c_{\ell}(n)$ denotes the number of $\ell$-core partitions of $n$, then we have (for example, see \cite{G-K-S, K}) the generating function
$$
\sum_{n=0}^{\infty} c_{\ell}(n)q^n=\prod_{n=1}^{\infty} \frac{(1-q^{\ell n})^{\ell}}{(1-q^n)}.
$$

\begin{example}
The Young diagram of the partition $\lambda:=(5,4,1),$ where each cell is labelled with its hook length, is given in Figure~1.
\begin{center}
\begin{ytableau}
  7 & 5 & 4 & 3 & 1\cr
  5 & 3 & 2 & 1\cr
  1 \cr
\end{ytableau}

\captionof{figure}{Hook lengths for $\lambda=(5,4,1)$}
\end{center}
\smallskip
By inspection, we see that $\lambda$ is an $\ell$-core for every prime $\ell >7.$
\end{example}
\smallskip

 To state the asymptotics formulas, we let $L\left(\leg{\cdot}{\ell},s\right)$ be the Dirichlet $L$-function
for the Legendre symbol $\leg{\cdot}{\ell},$ and let
\begin{equation}\label{alpha}
\alpha_{\ell}:= \frac{(2\pi)^{\frac{\ell-1}{2}}}{\left(\frac{\ell-3}{2}\right)! \cdot
\ell^{\frac{\ell}{2}}\cdot L\left(\leg{\cdot}{\ell},\frac{\ell-1}{2}\right)}.
\end{equation}
By the functional equations of these Dirichlet $L$-functions and the theory of generalized Bernoulli numbers, we have that $1/\alpha_{\ell}$ is always a positive integer (see p. 339 of \cite{G-O}). For example, we have
 $1/\alpha_5=1,$ $1/\alpha_7=8,$ $1/\alpha_{11}=1275,$ and  $1/\alpha_{13}=33463.$
 In addition, we require the integers $\delta_{\ell}:=(\ell^2-1)/24,$ and the
  {\it twisted Legendre symbol divisor functions}
 \begin{equation}\label{sigma}
\sigma_{\ell}(n):=\sum_{1\leq d\mid n} \leg{n/d}{\ell}d^{\frac{\ell-3}{2}}.
\end{equation}
In terms of these quantities and functions, we obtain the following asymptotics for $Z_{\ell}(n).$

\begin{theorem}\label{Theorem1}
 If $\ell\geq 5$ is prime, then as $n\rightarrow +\infty$ we have
 $$
Z_{\ell}(n)\sim \alpha_{\ell} \cdot \sigma_{\ell}(n+\delta_{\ell}) p(n)\gg_{\ell} n^{\frac{\ell-5}{2}} e^{\pi \sqrt{2n/3}}.
$$
\end{theorem}

\begin{remark}
Apart from a density zero subset, we have that $Z_\ell(n)=0$ when $\ell \in \{2, 3\}$ (see \cite{G-O}).
\end{remark}

As a corollary, we find that $Z(n)/p(n)$ grows faster than any power of $n.$
\begin{corollary}\label{Corollary1}
 If $d>0,$ then
$$
\lim_{n\rightarrow +\infty}\frac{Z(n)}{p(n)\cdot n^d}=+\infty.
$$
\end{corollary}

\begin{remark} We note that Corollary~\ref{Corollary1} is weaker than
$$Z(n)/p(n)^2 \gg 1/\log n,
$$
which can be found in the discussion after Lemma 2.3 of Peluse's paper \cite{P}.
\end{remark}

We turn to the problem of describing the zero entries in $\CC_n$ where both indices $\lambda$ and $\mu$ are $\ell$-core partitions.
For prime $\ell$ and large $n,$
the entries in $\CC_n$ indexed by $\ell$-core pairs $(\lambda, \mu)$ always have
$\chi_{\lambda}(\mu)=0.$

\begin{theorem}\label{Theorem2}
Suppose that $\ell$ is prime, and let $N_{\ell}:=(\ell^6- 2\ell^5 + 2\ell^4 -3\ell^2 + 2\ell)/24.$ If $n>N_{\ell}$ and $\lambda, \mu\vdash n$  are $\ell-$core partitions, then
$\chi_{\lambda}(\mu)=0.$
\end{theorem}

If $Z^*_{\ell}(n)$ denotes the number of vanishing entries $\chi_{\lambda}(\mu)=0$
indexed by $\ell$-core partitions $\lambda, \mu \vdash n,$ then we have the following corollary.

\begin{corollary}\label{Corollary2} For primes $\ell,$ the following are true.

\smallskip
\noindent
(1) Apart from a density zero subset, we have that $Z^*_\ell(n)=0$ when $\ell\in \{2, 3\}.$

\smallskip
\noindent
(2) If $\ell\geq 5$, then as $n\rightarrow +\infty$ we have
$$
Z^*_{\ell}(n)\sim \alpha_{\ell}^2 \cdot \sigma_{\ell}(n+\delta_{\ell})^2 \gg_{\ell} n^{\ell-3}.
$$
\end{corollary}

To obtain these results, we use the well-known vanishing result that follows from the Murnaghan–Nakayama rule and says that
$\chi_\lambda(\mu)=0$ whenever $\mu$ has a part that is not the length of any hook in $\lambda$. Therefore, our goal is reduced to
counting pairs of  partitions
$(\lambda, \mu)$ of large $n,$ where $\mu$ has a part that is a multiple of $\ell,$ and  where $\lambda$ is an $\ell$-core. 
Theorem~\ref{Theorem1} is obtained by estimating these counts using asymptotics and lower bounds for various partition functions due to Hardy and Ramanujan, Hagis, and Granville and the second author.

Theorem~\ref{Theorem2} concerns the cases where $(\lambda, \mu)$ are  both $\ell$-cores, and is a consequence of the fact (see Theorem~\ref{EleanorTheorem})
that every large $\ell$-core has a part that is a multiple of $\ell.$  This fact is proved using  the ``abacus theory'' of $\ell$-cores, and is a generalization of Section~3 of \cite{O-S} by Sze and the second author in the case of  $4$-core partitions.
Corollary~\ref{Corollary2} then follows from the asymptotics for $c_{\ell}(n)$ due to Granville and the second author.
 
 This paper is organized as follows.
Section~\ref{NutsAndBolts} recalls well-known vanishing result and bounds, as well as the asymptotics and estimates for the relevant partition functions. Section~\ref{Abacus} gives the abacus theory of $\ell$-cores and the statement and proof of Theorem~\ref{EleanorTheorem}.  In Section~\ref{Proofs} we employ these results to prove Theorems~\ref{Theorem1} and \ref{Theorem2}, and Corollaries~\ref{Corollary1} and ~\ref{Corollary2}. 

\section{Acknowledgements}
\noindent The authors thank Sarah Peluse and Richard Stanley as well as the referee for helpful comments that improved this paper.

\section{Nuts and Bolts}\label{NutsAndBolts}

In this section we recall essential facts  that we require for the proofs of our results.
We first state a criterion that guarantees the vanishing of character values, and then we give estimates for the relevant partition functions.

\subsection{Criterion for the vanishing of $\chi_{\lambda}(\mu)$}

Here we recall a standard partition theoretic criterion that guarantees the vanishing of a character value $\chi_{\lambda}(\mu)$. Suppose that $\lambda=(\lambda_1,\dots, \lambda_s)$ and
$\mu=(\mu_1, \dots, \mu_t)$ are partitions of size $n$, and let
  $\{h_{\lambda}(i,j)\}$ be the multiset of hook lengths for $\lambda.$ Thanks to the Murnaghan-Nakayama formula (for example, see Theorem 2.4.7 of \cite{J-K}),
we have that $\chi_{\lambda}(\mu)=0$ when $\{ \mu_i\}$ is not a subset of $\{h_{\lambda}(i,j)\}$.

Given a prime $\ell,$
this immediately gives natural families of vanishing character table entries indexed by
pairs of partitions $(\lambda, \mu)$ of $n$, where $\mu$ has a part that is a multiple of $\ell,$ and
$\lambda$ is an $\ell$-core partition.
To make use of this observation, we recall that a partition $\mu$ is  {\it   $A$-regular} if none of its parts $\mu_i$ are multiples of $A.$ 
If $p_{A}(n)$ denotes the number of $A$-regular partitions of $n$, then one easily confirms the generating function
$$
\sum_{n=0}^{\infty} p_{A}(n)q^n=\prod_{n=1}^{\infty} \frac{(1-q^{An})}{(1-q^n)}=
\prod_{n=1}^{\infty}\left (1+q^n+q^{2n}+\dots+q^{(A-1)n}\right),
$$
which shows that $p_{A}(n)$ also is the number of partitions of $n$ where parts appear at most 
 $A-1$ times.
In terms of $p(n), p_{\ell}(n)$ and $c_{\ell}(n)$, we have the following lower bounds for $Z(n).$

\begin{lemma}\label{Prop1} If $\ell$ is prime, then the following are true.

\smallskip
\noindent
(1)  If $\mu\vdash n$ is not an $\ell$-regular partition and $\lambda\vdash n$ is an $\ell$-core partition, then
$\chi_{\lambda}(\mu)=0.$

\smallskip
\noindent
(2) If $n$ is a positive integer, then we have
$$
Z_{\ell}(n)\geq \left(p(n)-p_{\ell}(n)\right)\cdot c_{\ell}(n).
$$
\end{lemma}
\begin{proof}

\noindent
(1) By hypothesis, $\mu$ is not $\ell$-regular, meaning that it has a part that is a multiple of $\ell.$
As $\lambda$ is an $\ell$-core, none of its hook lengths are multiples of $\ell.$ Therefore, $\chi_{\lambda}(\mu)=0$ by Murnaghan-Nakayama.
\smallskip
\noindent
(2) The number of partitions of $n$ that are not $\ell$-regular is $p(n)-p_{\ell}(n).$ Therefore, (1) gives the conclusion that
$$
Z_{\ell}(n) {\color{black}\ \geq \ } (p(n)-p_{\ell}(n))c_{\ell}(n).
$$ \end{proof}

\subsection{Estimates for some partition functions}\label{estimates}

Here we recall asymptotics and lower bounds for the partition functions we require to prove Theorem~\ref{Theorem1} and Corollary~\ref{Corollary2}. First we have the celebrated Hardy-Ramanujan asymptotic for $p(n).$

\begin{theorem}\label{HR}
As $n\rightarrow +\infty$, we have
$$
p(n)\sim \dfrac{1}{4n\sqrt{3}}\cdot \exp(\pi \sqrt{2n/3}).
$$
\end{theorem}

Hagis obtained asymptotics for $p_A(n),$ the number of $A$-regular partitions of $n.$ Letting $t=A-1$ in Corollary 4.2 of \cite{H}, we have the following asymptotic formula.

\begin{theorem}\label{Hagis}
If $A\geq 2,$ then we have
$$
p_A(n)=C_A (24n-1+A)^{-\frac{3}{4}}\exp\left(C\sqrt{\frac{A-1}{A}\left(n+\frac{A-1}{24}\right)}\right)
\left(1+O(n^{-\frac{1}{2}})\right),
$$
where $C:=\pi \sqrt{2/3}$ and $C_A:=\sqrt{12}A^{-\frac{3}{4}}(A-1)^{\frac{1}{4}}.$ 
\end{theorem}

Finally, we recall facts about $c_t(n),$ the number of $t$-core partitions of $n$ that were obtained by Granville and the second author in \cite{G-O}. In terms of $\alpha_{\ell}$ defined in (\ref{alpha}), and the twisted Legendre symbol divisor functions $\sigma_{\ell}(n)$ defined in (\ref{sigma}), we have the following theorem.

\begin{theorem}\label{GranvilleOno} The following are true.

\smallskip
\noindent
(1) We have that
$$
c_2(n)=\begin{cases} 1 \ \ \ \ \ \ &{\text {\rm if}}\ {\text {\rm $n$ is a triangular number,}}\\
0 \ \ \ \ \ \ &{\text {\rm otherwise}}.
\end{cases}
$$

\smallskip
\noindent
(2) If $n$ is a non-negative integer, then
$$
c_3(n)=\sum_{d\mid (3n+1)} \leg{d}{3}.
$$
In particular, $c_3(n)=0$ for almost all $n$.

\smallskip
\noindent
(3) If $t\geq 4$ and $n$ is a non-negative integer, then $c_t(n)>0.$

\smallskip
\noindent
(4) If $n$ is a non-negative integer, then
$c_5(n)=\sigma_5(n+1).$
\smallskip

\smallskip
\noindent
(5) If $\ell\geq 7$  is prime, then as $n\rightarrow +\infty$ we have
$$
c_{\ell}(n)\sim \alpha_{\ell}\cdot \sigma_{\ell}(n+\delta_{\ell}).
$$

\smallskip
\noindent
(6) If $\ell\geq 11$ is prime and $n$ is sufficiently large, then  we have
$$
c_{\ell}(n)> \frac{2\alpha_{\ell}}{5} \cdot n^{\frac{\ell-3}{2}}.
$$
\end{theorem}
\begin{proof}
Claim (1) is a straightforward observation. Claims (2), (4), and (5) are proved on p. 339-340 of \cite{G-O}. Claim (3) is Theorem~1 of \cite{G-O}, while
(6) is Theorem~4 of \cite{G-O}. 
\end{proof}

\section{Abaci and large $\ell$-core partitions}\label{Abacus}

Throughout this section, suppose that $\ell$ is prime.
The main result here is the following theorem which shows that every sufficiently large $\ell$-core partition has a part that is a multiple of $\ell.$

\begin{theorem}\label{EleanorTheorem}
Suppose that $\ell$ is prime, and let $N_{\ell}:=(\ell^6- 2\ell^5 + 2\ell^4 -3\ell^2 + 2\ell)/24 .$ If $n>N_{\ell},$ then every
$\ell$-core partition of size $n$ has a part that is a multiple of $\ell.$
\end{theorem}

\begin{remark} 
We note that $N_{\ell}<\ell^6/24$ is not optimal. Indeed, if we let $N^{\max}_{\ell}$ be the largest $n$ admitting an $\ell$-regular $\ell$-core partition, then it turns out that $N_3^{\max}=10$ and  $N_3=16.$
\end{remark}

\subsection{Abaci Theory} We make use of the theory of {\it abaci} for partitions (for example, see
\cite{E-M, J-K}). In particular, let $\lambda = \lambda_1 \geq \lambda_2 \geq \cdots \geq \lambda_s > 0$ be a partition of $n$. For each $1 \leq i \leq s,$ define the \textit{$i$th structure number} $B_i := \lambda_i -i+s$, so
that $B_i = h_\lambda(i,1),$ the hook length of cell $(i,1)$.

Using these structure numbers, we represent the partition $\lambda$ as an $\ell$-abacus $\mathfrak{A}_{\lambda},$ consisting of beads placed on rods numbered $0, 1,\dots, \ell-1.$
For each $B_i$, there is a
unique pair of integers ($r_i,c_i)$ for which
$B_i = \ell(r_i - 1) + c_i$ and $0 \leq c_i \leq \ell-1$. The abacus $\mathfrak{A}_{\lambda}$ then consists of $s$ beads, where for each $i$, one places a bead in position $(r_i,c_i)$.

\begin{lemma}[Lemma 2.7.13, \cite{J-K}]
Assuming the notation above, $\lambda$ is an $\ell$-core if and only if all of the beads in $\mathfrak{A}_{\lambda}$ lie at the top of their respective rods without gaps.
\end{lemma}

In view of this lemma, we may represent an abacus of an $\ell$-core partition by  $\ell$-tuples of non-negative integers, say $(b_0, \dots, b_{\ell-1}),$ where $b_i$ denotes the number of beads in column $i$. However, such representations are not unique as they generally allow for parts of size zero. We have the following elementary lemma.

\begin{lemma}[Lemma 1, \cite{O-S}]
The following abaci both represent the same $\ell$-core partition:
$$(b_0,b_1,\dots,b_{\ell-1}) \quad \text{and} \quad (b_{\ell-1} +1,b_0,b_1,\dots,b_{\ell-2}).$$
\end{lemma}

By repeatedly applying this lemma, we may canonically define the unique abacus representation for an $\ell$-core to be the one with zero beads in the first column. Thus, when we talk about \textit{the} abacus representation of an $\ell$-core $\lambda$, we will always mean the abacus of the form $\mathfrak{A}_{\lambda} = (0, b_1, \cdots, b_{\ell-1}).$\footnote{These abaci correspond to those representations of $\lambda$ without parts of size 0.}  Using these abaci, we offer the following lemma that will allow us to rule out the existence of partitions that are simultaneously $\ell$-core and $\ell$-regular for all but finitely many $n$.

\begin{lemma} \label{lem: beadjumps}
Suppose that $\mathfrak{A}_{\lambda} = (0, b_1, \dots, b_{\ell-1})$ is the abacus corresponding to an $\ell$-core $\lambda,$ and suppose that there is an integer $k\geq 0$ such that for each $1\leq i\leq \ell-1$ we have either $b_i\leq k$ or $b_i\geq k+\ell.$
If there is at least one $j$ for which $b_j\geq k+\ell,$ then
$\lambda$ is not an $\ell$-regular partition. 
\end{lemma}

\begin{remark}
	Let $\mathfrak{A}_{\lambda} = (0, b_1, \cdots, b_{\ell-1})$ be the abacus of an $\ell$-core $\lambda.$ If $\min(b_1, \cdots b_{\ell-1}) \geq \ell,$  then the proof of the lemma will show that $\lambda$ has a part of exact size
	 $\ell$. These are the cases where one can choose $k=0$ in the lemma.
\end{remark}

\begin{proof}
	By our hypothesis, we may fix $j$ for which $b_j \geq k + \ell$. Let $\delta$ denote the total number of columns with length at least $k+\ell$. Note that if $B_{i}$ and $B_{i'}$ are  structure numbers corresponding to consecutive beads in column $j$ between rows $k+1$ and $k+\ell$, then $|i-i'| = \delta$. Further, we have $|B_i - B_{i'}| = \ell$. 
	Generalizing the observation that $B_{i-1}-B_i = \lambda_{i-1}-\lambda_i + 1$, we have 
	\[|\lambda_{i} - \lambda_{i'}| = |B_{i} - B_{i'}| -\delta = \ell - \delta.
	\]
	
In particular, the difference between parts corresponding to consecutive beads in column $j$ between rows $k+1$ and $k+\ell$ is fixed and coprime to $\ell.$
As a consequence, these parts form a modulus $\ell-\delta$ arithmetic progression consisting of $\ell$ values. Thus, the parts cover all residue classes modulo $\ell$, and so includes a part that is a multiple of $\ell$.
\end{proof}

\begin{example} Let $\ell=3$, and consider the $3$-core abacus $(0, 4,1)$ as shown below.
\smallskip
$$\begin{tabular}{c|ccc}
1& $\cdot$ & $\circ$ & $\circ$ \\
2& $\cdot$ & $\circ$ & $\cdot$ \\
3& $\cdot$ & $\circ$ & $\cdot$ \\
4& $\cdot$ & $\circ$ & $\cdot$ \\
\end{tabular}$$
\smallskip

We illustrate Lemma \ref{lem: beadjumps} with $k=1.$ Since $b_1=3+1=4$ and $b_2 =1$, the lemma asserts that $\lambda$ has a part that is a multiple of 3. The structure numbers are found to be $B_1 = 10$, $B_2 = 7$, $B_3=4$, $B_4=2$, and $B_1=1$, and we compute that $\lambda_1 = 10+1-5 = 6$, $\lambda_2 = 4$, $\lambda_3 = 2$, and $\lambda_4 = \lambda_5 =1$. In particular, $\lambda_1$ is a multiple of 3. 

Finally, consider the abacus with the bead in row 4 removed.  One easily checks that the corresponding partition is $3$-regular, demonstrating that the condition on the size of the gap in column lengths cannot be relaxed. 
\end{example}

\smallskip
With  two more observations, we will be able to construct an abacus which gives an upper bound for the size of an $\ell$-regular $\ell$-core partition. 

First, suppose that $\lambda$ and  $\lambda'$ are $\ell$-cores with $\mathfrak{A}_\lambda =(0, b_1, \dots , b_{\ell-1})$ and $\mathfrak{A}_{\lambda'} =(0, b_1', \dots , b_{\ell-1}')$. Then we say $\mathfrak{A}_\lambda \leq \mathfrak{A}_{\lambda'}$ if $b_i \leq b_i'$ for all $1 \leq i \leq \ell-1$. This relation endows the set of $\ell$-core abaci with the structure of a directed partially ordered set. It is not hard to show that $\mathfrak{A}_\lambda \leq \mathfrak{A}_{\lambda'}$ implies $|\lambda| \leq |\lambda'|$.

For the purpose of obtaining $N_{\ell},$ the following lemma allows us to restrict our attention to those abaci where the $b_i$ are weakly increasing.

\begin{lemma} \label{lem: permute}
Suppose that $\lambda=(\lambda_1,\dots,\lambda_s)$ is an $\ell$-core partition of $n$ with abacus $\mathfrak{A}_{\lambda}=(0, b_1, \dots, b_{\ell-1}).$ If  there exist $1 \leq i < j \leq \ell-1$ for which $b_j < b_i,$ then the abacus $\mathfrak{A}'$ obtained by swapping $b_i$ and $b_j$ represents an $\ell$-core partition $\lambda'$ with $\lambda' \vdash n' >n$.
\end{lemma}

\begin{proof}
	We may write
	\begin{align*}
	n = \sum_{k=1}^s \lambda_k
	 = \sum_{k=1}^s (B_k + k -s)
	 = 	\sum_{k=1}^s B_k + \sum_{k=1}^s (k-s),
	\end{align*}
and likewise
$n' = \sum_{k=1}^s B'_k + \sum_{k=1}^s (k-s)$,
where $s$ remains the same because we have not changed the total number of beads. Since the second sum is the same in both expressions, it suffices to prove that 
$\sum_{i=1}^s B_k < \sum_{i=1}^s B'_k.$
Computing column-wise, we have 
\begin{align*}
	\sum_{k=1}^s B_k &= \sum_{m=1}^{b_i}\left(3(m-1)+i\right) + \sum_{m=1}^{b_j}\left(3(m-1)+j\right) + \underset{k \neq i,j}{\sum_{k=1}^{\ell-1}}\sum_{m=1}^{b_k} \left(3(m-1)+k\right) \\
	 &< \sum_{m=1}^{b_i}\left(3(m-1)+j\right) + \sum_{m=1}^{b_j}\left(3(m-1)+i\right) + \underset{k \neq i,j}{\sum_{k=1}^{\ell-1}}\sum_{m=1}^{b_k} \left(3(m-1)+k\right) \\
	 &=\sum_{k=1}^s B_k'
\end{align*}
as desired, where the inequality holds since $i<j$ and $b_i > b_j$. 
\end{proof}

\subsection{Proof of Theorem~\ref{EleanorTheorem}}

We aim to find an upper bound on $n$ such that $\lambda \vdash n$ can be an $\ell$-regular $\ell$-core partition. To do this, we will construct a partition $\Lambda$ such that $\lambda$ being an $\ell$-regular $\ell$-core implies $|\lambda| \leq |\Lambda|$. 
By Lemma \ref{lem: permute}, it suffices to restrict our attention to those $\ell$-cores whose abaci have weakly increasing column lengths. Suppose $\lambda$ is a weakly increasing $\ell$-regular $\ell$-core with abacus $\mathfrak{A}_\lambda = (0, b_1, \dots, b_{\ell-1})$. By Lemma \ref{lem: beadjumps}, we must have $\min(b_1, \dots, b_{\ell-1}) = b_1 \leq \ell-1$. By the same logic, we must have $b_i \leq i(\ell-1)$ for all $1\leq i \leq \ell-1$. 

Then if $\Lambda$ is the partition with abacus $\mathfrak{A}_{\Lambda} = (0, \ell-1, 2(\ell-1), \dots, (\ell-1)^2)$, we immediately have $\mathfrak{A}_\lambda \leq \mathfrak{A}_\Lambda$, which implies $|\lambda| \leq |\Lambda|$. Then $N_\ell := |\Lambda|$ gives an upper bound on $n$. By direct calculation, we find that
\[|\Lambda| = \sum_{i=1}^s (B_i +i-s)= \sum_{i=1}^{\ell-1} \sum_{j=1}^{i(\ell-1)} \left( \ell(j-1)+i \right) + \sum_{i=1}^{s} (i-s) = \frac{\ell^6- 2\ell^5 + 2\ell^4 -3\ell^2 + 2\ell}{24},
\]
where $s = \ell(\ell-1)^2/2$, giving the desired conclusion.

\section{Proofs of our results}\label{Proofs}
We are now in a position to  prove Theorems~\ref{Theorem1} and \ref{Theorem2}, and Corollaries~\ref{Corollary1} and ~\ref{Corollary2}.

\begin{proof}[Proof of Theorem~\ref{Theorem1}]
We note that Theorems~\ref{HR} and \ref{Hagis} imply that
$$
\lim_{n\rightarrow +\infty}\frac{p(n)-p_{\ell}(n)}{p(n)}=1.
$$
The claim now follows by combining Lemma~\ref{Prop1} (2), Theorem~\ref{HR} , and Theorem~\ref{GranvilleOno} (4-6).
\end{proof}

\begin{proof}[Proof of Corollary~\ref{Corollary1}]
This claim follows from Theorem~\ref{Theorem1} by choosing primes $\ell \rightarrow +\infty.$
\end{proof}

\begin{proof}[Proof of Theorem~\ref{Theorem2}]
By Theorem~\ref{EleanorTheorem}, every $\ell$-core partition of size $n>N_{\ell}$ has a part that is a multiple of  $\ell.$  Since every hook length of an $\ell$-core is not a multiple of $\ell$,  it follows from Murnaghan-Nakayama that
whenever $\lambda, \mu \vdash n$ are $\ell$-cores with $n>N_{\ell},$
we have $\chi_{\lambda}(\mu)=0.$
\end{proof}

\begin{proof}[Proof of Corollary~\ref{Corollary2}]
Thanks to Theorem~\ref{Theorem2}, we have that $Z^*_{\ell}(n) =c_{\ell}(n)^2$ for sufficiently large $n$.
The claimed asymptotics and inequalities follow from Theorem~\ref{GranvilleOno} (4-6).
\end{proof}

\end{document}